\documentclass{article}
\usepackage{color}
\usepackage{eurosym}
\usepackage{amsmath}
\usepackage{amsfonts}
\usepackage{verbatim}
\usepackage{graphicx}
\usepackage[noend]{algpseudocode}
\usepackage{caption}
\usepackage{subcaption}
\usepackage{epstopdf}

\newtheorem{theorem}{Theorem}

\newtheorem{algorithm}[theorem]{Algorithm}

\newtheorem{definition}[theorem]{Definition}

\newtheorem{lemma}[theorem]{Lemma}

\newtheorem{proposition}[theorem]{Proposition}
\newtheorem{remark}[theorem]{Remark}

\newenvironment{proof}[1][Proof]{\noindent\textbf{#1.} }{\ \rule{0.5em}{0.5em}}

\input{tcilatex}
\begin{document}

\title{Rigorous computation of invariant measures and fractal dimension for
piecewise hyperbolic maps: 2D Lorenz like maps.}
\author{Stefano Galatolo \footnote{ Dipartimento di Matematica, Universita di Pisa, Via \ Buonarroti 1,Pisa. Email: galatolo@dm.unipi.it}, Isaia Nisoli \footnote{ Instituto de Matem\'{a}tica - UFRJ
 Av. Athos da Silveira Ramos 149,
Centro de Tecnologia - Bloco C
Cidade Universitária -
Ilha do Fund\~{a}o.
Caixa Postal 68530
21941-909 Rio de Janeiro - RJ - Brasil   Email:nisoli@im.ufrj.br}}

\maketitle

\begin{abstract}
We consider a class of piecewise hyperbolic maps from the unit square to
itself preserving a contracting foliation and inducing a piecewise expanding
quotient map, with infinite derivative (like the first return maps of
Lorenz like flows). We show how the physical measure of those systems can be
rigorously approximated with an explicitly given bound on the error, with
respect to the Wasserstein distance. We apply this to the rigorous
computation of the dimension of the  measure. We
present a rigorous implementation of the algorithms using interval
arithmetics, and the result of the computation on a nontrivial example of
Lorenz like map and its attractor, obtaining a statement on its local dimension.
\end{abstract}

\tableofcontents

\section{Introduction}

\paragraph{\textbf{Overview}}

Several important features of the statistical behavior of a dynamical system
are \textquotedblleft encoded\textquotedblright\ in the so called \emph{%
Physical Invariant Measure\footnote{%
Physical invariant measures are the ones which (in some sense that will be
precised below) represent the statistical behavior of a large set of initial
conditions.}}. The knowledge of the invariant measure can give information
on the statistical behavior for the long time evolution of the system. This
strongly motivates the search for algorithms which are able to compute
quantitative information on invariant measures of physical interest, and in
particular algorithms giving an explicit bound on the error which is made in
the approximation.

The problem of approximating the invariant measure of dynamical systems was
broadly studied in the literature. Some algorithm is proved to converge to
the real invariant measure in some classes of systems (up to errors in some
given metrics), but results giving an explicit (rigorous) bound on the error
are relatively few, and really working implementations, even fewer (e.g. 
\cite{BB,GN,H,KMY,L}). Almost every (rigorous) implementation and almost all
methods works in the case of one dimensional or expanding maps.

The case where contracting directions are present does not easily fit with
known techniques, based on the choiche of a suitable functional analytic
framework and on the related spectral properties of the transfer operator
(or on Hilbert cones), because the involved functional spaces and the needed
a priori estimations are not easy to be brought in the form which is
necessary for an effective implementation. 

The output of a computation with an explicit estimation for
the error can be seen as a rigorously (computer aided) proved statement, and
hence has a mathematical meaning. In our case, the rigorous approximation
for invariant measures gives us the possibility to have rigorous
quantitative estimations on some aspects of the statistical and geometrical
behavior of the system we are interested in. In particular we will use it to have a statement on the
dimension of its physical invariant measure.

About the general problem of computing invariant measures, it is worth to
remark that some negative result are known. In \cite{GalHoyRoj3} it is shown
that\emph{\ there are examples of computable\footnote{%
Computable, here means that the dynamics can be approximated at any accuracy
by an algorithm, see e.g. \cite{GalHoyRoj3} for precise definition.} systems
without any computable invariant measure}. This shows some subtlety in the
general problem of computing the invariant measure up to a given error.

In this paper we focus on a class of Lorenz like maps, which are piecewise
hyperbolic maps with unbounded derivatives preserving a \emph{contracting
foliation}, similar to the Poincar\'{e} map of the famous Lorenz system.

We consider maps $F$ acting on $Q=I\times I$ \ (where $I=[-\frac{1}{2},\frac{%
1}{2}]$) having the following properties:

\begin{description}
\item[1)] $F:\Sigma \rightarrow \Sigma $ \ is of the form $%
F(x,y)=(T(x),G(x,y))$ (preserves the natural vertical foliation of the
square) and:

\item[2)] $T:I\rightarrow I$ is\ onto and piecewise monotonic, with $N,$
increasing, expanding branches with possibly infinite derivative: there are $%
c_{i}\in \lbrack 0,1]$ for $0\leq i\leq N$ with $0=c_{0}<\cdots <c_{N}=1$
such that $T|_{(c_{i},c_{i+1})}$ is continuous and monotone for $0\leq i<N$.
Furthermore, for $0\leq i<N$, $T|_{(c_{i},c_{i+1})}$ is $C^{1}$ and $\inf_{x\in
P}|T^{\prime }(x)|>1$.

\item[3)] $F$  is
uniformly contracting on each vertical leaf $\gamma $: there is a $\lambda <1$ such that
$|G(x,y_{1})-G(x,y_{2})|\leq \lambda \cdot |y_{1}-y_{2}|$;

\item[4)] \label{it:item_2} 
 $G:Q\rightarrow (0,1)$ is $C^{1}$ on $P\times
\lbrack 0,1]$, where $P=[0,1]\setminus \cup _{0\leq i<N}c_{i}$. 
 Furthermore, $\sup |\partial G/\partial x|<\infty $ and $|(\partial G/\partial y)(x,y)|>0$ for $%
(x,y)\in P\times \lbrack 0,1]$; 

\item[5)] $\frac{1}{|T^{\prime }|}$ has bounded variation.
\end{description}

About the regularity of $T$: we suppose that $\frac{1}{|T^{\prime }|}$ has
bounded variation to simplify the computation of the invariant measure of
this induced map. We remark that in general, for Lorenz like systems this
assumption should be replaced by generalized bounded variation (see \cite%
{AGP,keller}). This kind of maps however still satisfy a Lasota-Yorke
inequality, and the general strategy for the computation of the invariant
measure should be similar to the one used here and explained in Section \ref%
{1d} for the bounded variation case.

We approach the computation of the invariant measure for the two dimensional map by some techniques which have been
succesfully used to estimate decay of correlations in systems preserving a
contracting foliation (see \cite{AGP,GP10}). 
In these systems, the physical invariant measure can be seen as the limit of
iterates of a suitable absolutely continuous initial measure.
Our strategy, in order to compute this measure with an explicit bound on the
error, is to iterate a suitable initial measure a sufficient number of times
and carefully estimate the speed with which it approaches to the limit. This is
not sufficient for the computation since there is a further technical
problem: \emph{the computer cannot perfectly simulate a real iteration}.
Thus we need to understand how far simulated iterates are from real iterates.

Hence the algorithm and the estimation of the error involve two main steps:

\begin{description}
\item[a)] we estimate how many iterates of a suitable starting measure%
\footnote{%
The suitable measure to be iterated is constructed starting from a $L^{1}$
approximation of the absolutely continuous invariant measure of the induced
map $T$.} in the real system are necessary to approach the invariant measure
at a given distance (see Theorem \ref{uno}), and then

\item[b)] we estimate the distance between real iterates and the iterates of
a suitable discretized model which can be implemented on a computer (see
Proposition \ref{2}).
\end{description}

Altogether this allows  to implement an algorithm which rigorously
approximates the invariant measure by a suitable discretization of the
system (in the paper we will consider the so called Ulam discretization
method, which approximate the system by a Markov chain).

The results and the implementations which are presented are meant as a proof
of concept, to solve the problem and run experiments in some nontrivial and
interesting class of examples. We expect that a very similar strategy apply
in many other cases of systems preserving a contracting foliation.

In the next sections we describe more precisely the problem and the
technical tools we use to approach it: in section \ref{sec2} we introduce
some basic tools which are used in our construction.

In section \ref{lagi} and \ref{approx} we show the general mathematical
estimates which allows to implement the above two main steps a), b).

We then describe informally the algorithm which is meant to be implemented,
and then in Section \ref{dim} we show how, by the approximated knowledge of
the invariant measure and of the geometry of the system, it is possible to
approximate its local dimension.

In section \ref{8} we describe the implementation of the algorithm and some
remarks which permitted us to optimize it.

The rigorous implementation of our algorithm is substantially made by
interval arithmetics. It presents several technical issues; as an example we
mention that since the map is two-dimensional the number of cells involved
in the discretization increases, a priori, as the square of the size of the
discretization. This seriously affect the speed of the computation and the
possibility to reach a good level of precision. The presence of the
contracting direction, and an attractor which is not two dimensional allows
to find a suitable reduction of the discretization (restricting computations
to a neighborhood of the attractor) which reduces the complexity of the
problem (see Section \ref{subsec:rid}).

In Section \ref{meas2d} we show the result of the computation of the
invariant measure on an example of two dimensional Lorenz like map.

The computation of the invariant measure also allows the rigorous
approximation of the local dimension of the measure we are interested in. In
Section \ref{dimension} we show the result of the computation of the
dimension of a non trivial example.

\section{The general framework\label{sec2}}

In the next subsections we explain some preliminary notions and results used
in the paper.

\paragraph{The transfer operator}

Let us consider the space $SM(X)$ of Borel measures with sign on $X.$ A
function $T$ between metric spaces naturally induces a linear function $%
L_{T}:SM(X)\rightarrow SM(X)$ called the \textbf{transfer operator}
(associated to $T$) which is defined as follows. If $\mu \in SM(X)$ then $%
L_{T}[\mu ]\in SM(X)$ is the measure such that

\begin{equation*}
L_{T}[\mu ](A)=\mu (T^{-1}(A)).
\end{equation*}%
Sometimes, when no confusion arises, we will denote $L_{T}$ more simply by $%
L $.

Measures which are invariant for $T$ are fixed points of $L$, hence the
computation of invariant measures very often is done by computing some fixed
points of this operator. The most applied and studied strategy is to find a
finite dimensional approximation for $L$ (restricted to a suitable function
space) reducing the problem to the computation of the corresponding relevant
eigenvectors of a finite matrix. In this case some quantitative stability
result may ensure that the fixed point of the approximated operator is near
to the real fixed point which was meant to be computed (see Section \ref{1d}
for one example).

On the other hand, in many other interesting cases the invariant measure can
be computed as the limit of the iterates of some suitable starting measure $%
\overline{\mu} =\lim_{n\rightarrow \infty }L^{n}(\mu _{0})$. To estimate the error of
the approximation is important to estimate the speed of convergence (in some
topology). Another strategy is then to iterate the finite dimensional
approximating operator a suitable number of times to ``follow'' the
iterations of the original operator which will converge to the fixed point.

In this paper we consider a class of maps preserving a contracting foliation.
For this kind of maps it is possible to compute the speed of convergence of
suitable measures to the invariant one (see Section \ref{lagi}) and this is
the main idea we apply to compute the invariant measure. A suitable starting
measure has however to be computed. This is done by observing that this kind
of maps induces a one dimensional map representing the dynamics between the
leaves. We approximate the physical invariant measure for this map (up to
small errors in $L^{1}$, this will be done by a suitable fixed point
stability result, see Section \ref{1d}) then we use this approximation to
construct a suitable invariant measure to be iterated.

\paragraph{The Ulam method.}

We now describe a finite dimensional approximation of $L$ which is useful to
approximate invariant measures in the $L^{1}$ norm (see e.g. \cite%
{DelJu02,Din94,F08,GN,L,KMY}), and as we will see it also works with the
Wasserstein distance in our case.

Let us suppose now that $X$ is a manifold with boundary. Let us describe 
\emph{Ulam's Discretization} method. In this method the space $X$ is
discretized by a partition $I_{\delta }$ (with $k$ elements) and the system
is approximated by a finite state Markov Chain with transition probabilities 
\begin{equation}
P_{ij}={m(T}^{-1}{(I_{j})\cap I_{i})}/{m(I_{i})}  \label{pij}
\end{equation}%
(where $m$ is the normalized Lebesgue measure on $X$) and defining a
corresponding finite-dimensional operator $L_{\delta }$ ($L_{\delta }$
depend on the whole chosen partition but simplifying we will indicate it
with a parameter $\delta $ related to the size of the elements of the
partition) we remark that in this way, to $L_{\delta }$ it corresponds a
matrix $P_{k}=(P_{ij})$ .

Alternatively $L_{\delta }$ can be seen in the following way: let $F_{\delta
}$ be the $\sigma -$algebra associated to the partition $I_{\delta }$, then:%
\begin{equation}
L_{\delta }(f)=\mathbf{E}(L(\mathbf{E}(f|F_{\delta }))|F_{\delta })
\label{000}
\end{equation}%
where $\mathbf{E}$ is the conditional expectation. Taking finer and finer
partitions, in certain systems including for example piecewise expanding
one-dimensional maps, the finite dimensional model converges to the real one
and its natural invariant measure to the physical measure of the original
system.

We use the Ulam discretization both when applying the fixed point stability
result to compute the one dimensional invariant measure necessary to start
the iteration, and when constructing an approximated operator to iterate the two dimensional
starting measure.

\paragraph{The Wasserstein distance\label{sec:W-Kdistance}}

We are going to approximate the interesting invariant measure of our Lorenz
like map up to small errors in the Wasserstein metric.

If $X$ is a metric space, we denote by $SM(X)$ the set of Borel finite
measures with sign on $X$. Let $g:X\to\mathbb{R}$; let 
\begin{equation*}
L(g):=\sup_{x,y}\frac{|g(x)-g(y)|}{|x-y|}
\end{equation*}
be the best Lipschitz constant of $g$ and set $\Vert g\Vert _{\text{Lip}%
}=\Vert g\Vert _{\infty }+L(g).$

Let us consider the following slight modification of the classical notion of
Wasserstein distance between probability measures: given two measures $\mu
_{1}$ and $\mu _{2}$ on $X$, we define their  distance as

\begin{equation*}
W(\mu _{1},\mu _{2})=\underset{g~s.t.~L(g)\leq 1,||g||_{\infty }\leq 1}{\sup 
}|\int_{X}g~d\mu _{1}-\int_{X}g~d\mu _{2}|.
\end{equation*}

When $\mu _{1}$ and $\mu _{2}$ are probability measures, this is equavalent
to the classical notion.

Let us denote by $||.||$ the norm relative to this notion of distance%
\begin{equation*}
||\mu ||=\sup_{\substack{ \phi \in 1-\text{Lip}(I)  \\ ||\phi ||_{\infty
}\leq 1}}|\int \phi ~d\mu |.
\end{equation*}

\begin{remark}
By definition it follows that if $\mu =\sum_{1}^{m}\mu _{i}$, $\nu
=\sum_{1}^{m}\nu _{i}$ 
\begin{equation}
W(\mu ,\nu )\leq \sum_{1}^{m}W(\mu _{i},\nu _{i}).
\end{equation}%
Moreover, (see \cite{GP10}) if $\mu $ and $\nu $ are probability measures,
and $F$ is a $\lambda $-contraction ($\lambda <1$), then%
\begin{equation*}
W(L_{F}(\mu ),L_{F}(\nu ))\leq \lambda \cdot W(\mu ,\nu ).
\end{equation*}
\end{remark}

\section{Systems with contracting fibers, disintegration and effective
estimation for the speed of convergence to equilibrium.}

\label{lagi}

As explained before, we want to estimate how many iterations are needed for
a suitable starting measure supported on a neighborhood of the attractor, to
approach the invariant measure. This kind of estimation is similar to a
decay of correlation one, and we use an approach similar to the one used in 
\cite{AGP} to prove exponential decay of correlation for a class of systems
with contracting fibers. Here a more explicit and sharper estimate is needed.

Let us introduce some notations: we will consider the $\sup $ distance on
the square $Q=[-\frac{1}{2},\frac{1}{2}]^{2}$, so that the diameter, $\text{%
Diam}(Q)=1$. This choice is not essential, but will avoid the presence of
some multiplicative constants in the following, making notations cleaner.

The square $Q$ will be foliated by stable, vertical leaves. We will denote
the leaf with $x$ coordinate by $\gamma _{x}$ or, with a small abuse of
notation when no confusion is possible, we will denote both the leaf and its
coordinate with $\gamma $.

Given a measure $\mu $ and a function $f$, let $f\mu $ be the measure $\mu
_{1}$ such that $d\mu _{1}=fd\mu $. Let $\mu $ be a measure on $Q$. In the
following, such measures on $Q$ will be often disintegrated in the following
way: for each Borel set $A$%
\begin{equation}
\mu (A)=\int_{\gamma \in I}\mu _{\gamma }(A\cap \gamma )d\mu _{x}
\label{dis}
\end{equation}%
with $\mu _{\gamma }$ being probability measures on the leaves $\gamma $ and 
$\mu _{x}$ is the marginal on the $x$ axis which will be an absolutely
continuous measure.

Let us consider a Lorenz like two dimensional map $F$ and estimate
explicitly the speed of convergence of iterates of two initial measures with
absolutely continuous marginal.

\begin{theorem}
\label{uno}Let $F:Q \rightarrow Q $ as above, let $\mu ,\nu \in PM(\Sigma )$
be two measures with absolutely continuous marginals $\mu _{x},\nu _{x}$.
Then%
\begin{equation*}
W(L_{F}^{n}(\mu ),L_{F}^{n}(\nu ))\leq \lambda ^{n}+||\mu _{x}-\nu
_{x}||_{L^1}.
\end{equation*}%
Where we recall that $\lambda $ is the contraction rate on the vertical
leaves.
\end{theorem}

In the proof we use the following, proposition (see \cite{AGP}, Proposition
3) which allows to estimate the Wasserstein distance of two measures by its
disintegration on stable leaves.

\begin{proposition}
\label{prod}Let $\mu ^{1}$, $\mu ^{2}$ be measures on $Q$ as above, such
that for each Borel set $A$ 
\begin{equation*}
\mu ^{1}(A)=\int_{\gamma \in I}\mu _{\gamma }^{1}(A\cap \gamma )d\mu
_{x}^{1}~and~\mu ^{2}(A)=\int_{\gamma \in I}\mu _{\gamma }^{2}(A\cap \gamma
)d\mu _{x}^{2},
\end{equation*}%
where $\mu _{x}^{i}$ is absolutely continuous with respect to the Lebesgue
measure. In addition, let us assume that

\begin{enumerate}
\item $\int_{I}W(\mu _{\gamma }^{1},\mu _{\gamma }^{2})d\mu _{x}^{1}\leq
\epsilon $\label{it:well_defined}

\item $V(\mu _{x}^{1},\mu _{x}^{2})\leq \delta $ (where $V(\mu _{x}^{1},\mu
_{x}^{2})=\sup_{|g|_{\infty }\leq 1}|\int gd\mu _{x}^{1}-\int gd\mu
_{x}^{2}| $ is the \textbf{total variation distance}).
\end{enumerate}

Then $|\int gd\mu ^{1}-\int gd\mu ^{2}|\leq ||g||_{\text{Lip}}\cdot(\epsilon
+\delta ).$
\end{proposition}

\begin{remark}
\label{rem:well_defined} Referring to Item \ref{it:well_defined} we garantee
the left hand side to be well defined by assuming (without changing $\mu
^{2} $) that $\mu _{\gamma }^{2}$ is defined in some way, for example $\mu
_{\gamma }^{2}=m$ (the one dimensional Lebesgue measure on the leaf)\ for
each leaf where the density of $\mu _{x}^{2}$ is null.
\end{remark}

\begin{proof}[Proof of Theorem \protect\ref{uno}]
Let us consider $\{I_{i}\}_{i=1,...,m}$ the intervals where the branches of $%
T^{n}$ are defined. Let us consider $\varphi _{i}=1_{I_{i}\times I}$ and let 
\begin{equation*}
\mu_{i}=\varphi _{i}\mu \quad \nu _{i}=\varphi _{i}\nu,
\end{equation*}
then $\mu =\sum \mu _{i}$, $\nu =\sum \nu _{i}$; thus by triangle inequality%
\begin{equation*}
W(L_{F}^{n}(\mu ),L_{F}^{n}(\nu ))\leq \sum_{i=1,..,m}W(L_{F}^{n}(\mu
_{i}),L_{F}^{n}(\nu _{i})).
\end{equation*}

Let us denote by $T_{i}:=T^{n}|_{I_{i}}$, remark that this is injective and
recall that $T^{n}$ is a $L^{1}$ contraction. Then by Proposition \ref{prod} 
\begin{eqnarray*}
W(L_{F}^{n}(\mu _{i}),L_{F}^{n}(\nu _{i})) &\leq &\int_{I}W((L^{n}\mu
_{i})_{\gamma },(L^{n}\nu _{i})_{\gamma })~dL_{T^{n}}((\nu
_{i})_{x})+||L_{T}^{n}((\mu _{i})_{x})-L_{T}^{n}((\nu _{i})_{x})||_{L^1} \\
&\leq &\int_{I}W(L_{F}^{n}((\mu _{i})_{T_{i}^{-1}(\gamma )}),L_{F}^{n}((\nu
_{i})_{T_{i}^{-1}(\gamma )}))~dL_{T^{n}}((\nu _{i})_{x})+||(\mu
_{i})_{x}-(\nu _{i})_{x}||_{L^1} \\
&\leq &\lambda ^{n}\int_{I}W((\mu _{i})_{T_{i}^{-1}(\gamma )},(\nu
_{i})_{T_{i}^{-1}(\gamma )})~dL_{T^{n}}((\nu _{i})_{x})+||(\mu
_{i})_{x}-(\nu _{i})_{x}||_{L^1} \\
&=&\lambda ^{n}\int_{I_{i}}W((\mu _{i})_{\gamma },(\nu _{i})_{\gamma
})~d((\nu _{i})_{x})+||(\mu _{i})_{x}-(\nu _{i})_{x}||_{L^1}.
\end{eqnarray*}

Summarizing: 
\begin{eqnarray*}
W(L_{F}^{n}(\mu ),L_{F}^{n}(\nu )) &\leq &\lambda
^{n}\sum_{i}\int_{I_{i}}W((\mu _{i})_{\gamma },(\nu _{i})_{\gamma })~d((\nu
_{i})_{x})+||(\mu _{i})_{x}-(\nu _{i})_{x}||_{L^1} \\
&=&\lambda ^{n}\int_{I}W(\mu _{\gamma },\nu _{\gamma })~d(\nu _{x})+||\mu
_{x}-\nu _{x}|||_{L^1}.
\end{eqnarray*}
\end{proof}

The two dimensional map $F$ induces a one dimensional one $T$ which is
piecewise expanding. In this kind of maps the application of the Ulam method
with cells of size $\delta $, gives a way to approximate the absolutely
continuous invariant measure $f$ of $T$ by a step function $f_{\delta }$,
which is the steady state of the associated Markov chain (see Section \ref%
{1d} for the details). We will use $f_{\delta }$ to construct a suitable
starting measure for the iteration process. Denoting by $\overline{\mu }$
the physical invariant measure of $F$, we will consider the above iteration
process by iterating $\overline{\mu }$ and a starting measure $\mu _{0}$
supported on a suitable open neighborhood $U$ of the attractor and having $%
f_{\delta }$ as marginal on the expanding direction.

\section{Approximated 2D iterations}

\label{approx}

The general idea is to approach the invariant measure by iterating a
suitable measure. We remark that we cannot simulate real iterates on a
finite computer, and we can only work with approximated iterates. We will
then estimate the distance between these and real ones.

Let us consider the grid partition $\mathcal{Q}=\{Q_{i,j}\}$ of $Q$,
dividing $Q$ in rectangles $Q_{i,j}$ of size $\delta \times \delta
^{^{\prime }}$ (where $\delta ,\delta ^{^{\prime }}$ are the inverses of two
integers). Let $I_{i}$ be the relative subdivision of $[0,1]$ in $\delta
^{\prime }$ long segments. We also denote $I_{-1}=I_{\frac{1}{\delta
^{\prime }}+1}=\emptyset $. Let $\pi _{1}$ and $\pi _{2}$ be the two,
natural projections of $Q$, respectively along the vertical and horizontal
direction.

Let $Q_{i}=\cup _{j}Q_{i,j}$ be the horizontal, row strips and let $%
Q_{j}^{t}=\cup _{i}Q_{i,j}$ be the vertical ones. Moreover let $\varphi
_{i}=1_{Q_{i}}$ be the indicator function of $Q_{i}$; these are $\frac{1}{%
\delta ^{\prime }}$ many functions as shown in figure \ref{fig:quadrati}.

\begin{figure}[h]
\centering
\includegraphics[width=60mm]{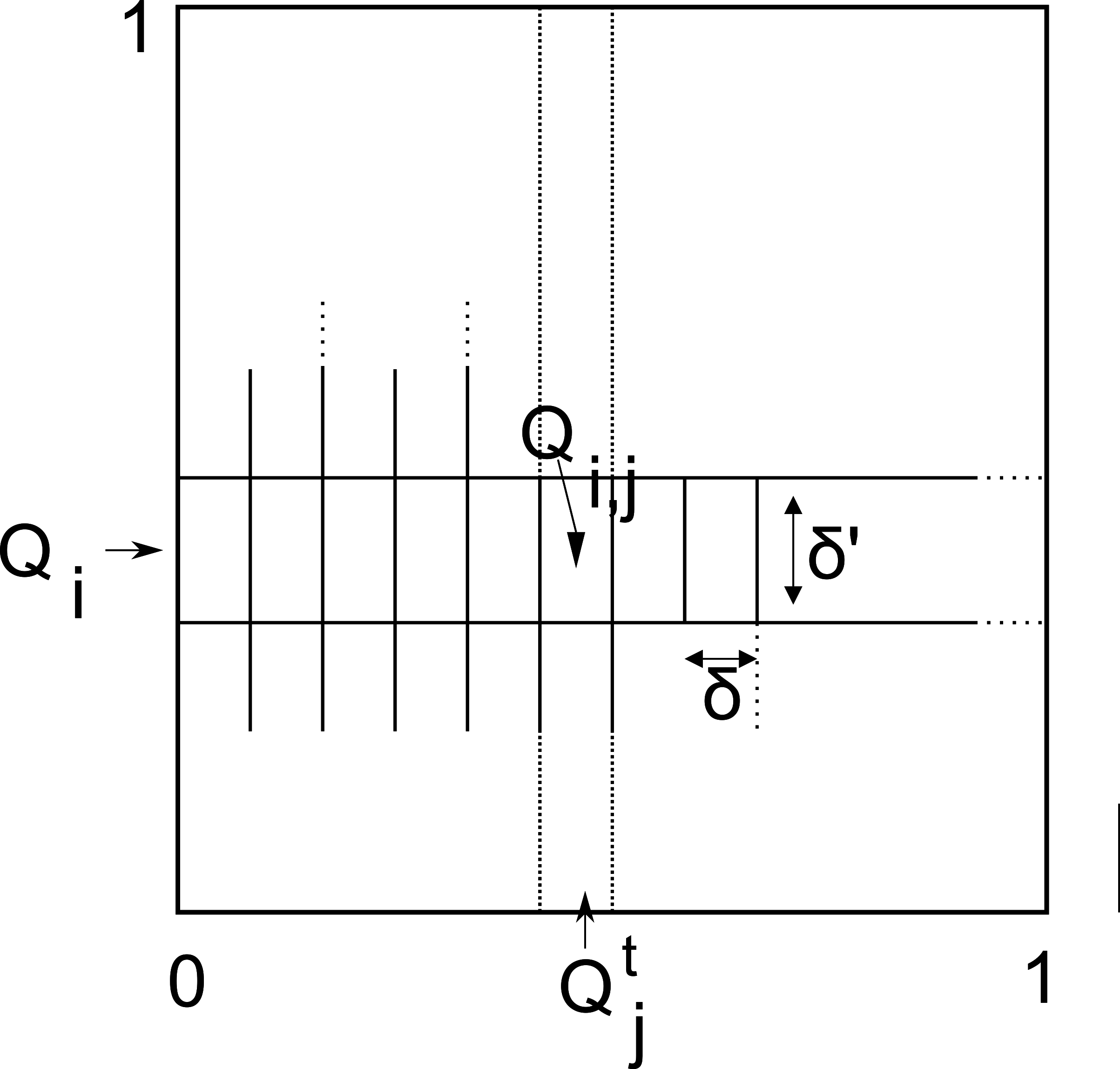} 
\caption{The partitioning scheme and the notation}
\label{fig:quadrati}
\end{figure}


Let us denote by $\varphi _{i}m$, the measure having density $\varphi _{i}$
(with respect to the Lebesgue one).

We will also need to perform some construction on measures. Let us introduce
some notations which will be helpful.

Given a measure $\mu $, let us now consider two projection operators
averaging on vertical or horizontal segments, $P_{\_}$ and $%
P_{|}:PM(Q)\rightarrow PM(Q)$ defined by

\begin{equation*}
P_{|}\mu =\sum_{i}\pi _{1}^{\ast }(\varphi _{i}\mu )\times \pi _{2}^{\ast
}(\varphi _{i}m)
\end{equation*}%
and $P_{\_}\mu $ the measure obtained similarly, averaging on horizontal
segments in $Q_{i,j}$: 
\begin{equation*}
P_{\_}\mu =\sum_{i,j}\pi _{1}^{\ast }(1_{Q_{i,j}}m)\times \pi _{2}^{\ast
}(1_{Q_{i,j}}\mu )
\end{equation*}%
so that $P_{\_}P_{|}\mu =\sum_{i,j}\mu (Q_{i,j})1_{Q_{i,j}}m=\mathbf{E}(\mu
|\{Q_{i,J}\})$ and $\pi _{1}(P_{\_}P_{|}\mu )=\mathbf{E}(\pi _{1}(\mu
)|\{I_{i}\})$.

Let us define 
\begin{equation*}
L_{\delta }=P_{-}P_{|}LP_{-}P_{|}.
\end{equation*}%
This is a finite rank operator and is the Ulam discretization of $L$ with
respect to the rectangle partition.

We remark that $L$ is not a contraction on the $W$ distance, to realize it,
consider a pair of Dirac-$\delta $-measures on the expanding direction. This
is a problem, in principle, when simulating real iterations of the system by
approximate ones. The problem can be overcome disintegrating the measure
along the stable leaves and \emph{exploiting the fact that the measures we
are interested in, are absolutely continuous on the expanding direction and
the system, in some sense, will be stable for this kind of measures}. This
can be already noticed in Theorem \ref{uno} where it can be seen that the
(total variation) distance between the marginals does not increase by
iterating the transfer operator.

Now let us define the  measure $\mu _{0}$ \ which is meant to be
iterated and estimate $W(L^{n}\mu _{0},L_{\delta }^{n}\mu _{0})$. Let us
consider the physical invariant measure of the system $\overline{\mu }$ and
iterate a starting measure $\mu _{0}$ supported on a suitable open
neighborhood $U$ of the attractor\footnote{%
This neighborhood will be constructed in the implementation by intersecting
\ $F^{n}(Q)$ with a suitable grid and taking all the rectangles with non
empty intersection.}. We suppose that $U$ \ is such that \ $U\cap I_{\gamma
} $ is a finite union of open intervals,\ where $I_{\gamma }=\{(x,y)\in
\lbrack 0,1]\times \lbrack 0,1],x=\gamma \}$ is a vertical leaf at
coordinate $\gamma $. Given $U$, we construct $\mu _{0}$ in a way that it
has the computed approximation $f_{\delta }$ of the one dimensional invariant mesure, as marginal on the expanding direction. We also construct
the measure $\mu _{0}$ in a way that there is on each stable leaf, a
multiple of the Lebesgue measure $m_{U_{\gamma }}$ on the union of intervals 
$U_{\gamma }=$\ $U\cap I_{\gamma }$.

More precisely

\begin{equation}  \label{eq:mu_0}
\mu _{0}=\ f_{\delta }\times \frac{m_{U_{\gamma }}}{m_{U_{\gamma
}}(I_{\gamma })}.
\end{equation}

\begin{proposition}
\label{2}Let us consider a Lorenz like map $F$ as described in the
introduction, its transfer operator $L$ and the finite dimensional Ulam
approximation $L_{\delta }$ with grid size $(\delta ,\delta ^{\prime })$ a
described above. Let $L_{T}$ be the one dimensional transfer operator
associated to the action of $L$ on the $x-$marginals. Let $\ \mu _{0}$
described above, and let $\epsilon \geq ||f-f_{\delta }||_{L^1}$, let moreover
suppose that the whole space can be divided into two sets $X_{1}=X-\mathcal{B%
}\times I,X_{2}=$ $\mathcal{B}\times I$ for some finite union of intervals $%
\mathcal{B}\subset I$ such that $\mu _{0}(X_{2})\leq l$ (we have a bound for
the measure of the bad part of the space) and $Lip(F|_{X_{1}})\leq \overline{%
L}$ . Then it holds for each $n$ 
\begin{gather}
W(L^{n}\mu _{0},L_{\delta }^{n}\mu _{0})\leq \frac{2\delta ^{\prime }}{%
1-\lambda }+\delta  \label{err} \\
+\sum_{i=1}^{n-1}\min \{[\overline{L}^{i}(\delta +2\delta ^{\prime })+(%
\overline{L}^{i-1}+\overline{L}^{i-2}+...+1)(2\delta ^{^{\prime
}}+2l+3\epsilon )],  \notag \\
\lbrack \lambda ^{n-i}+\frac{2\delta ^{^{\prime }}}{1-\lambda }+||f_{\delta
}-L_{T}f_{\delta }||_{L^{1}}]\}.  \notag
\end{gather}
\end{proposition}

\begin{remark}\label{r6}
We remark that that since $f_{\delta }$ is known, $||f_{\delta }||_{BV}$ can
be recursively estimated, moreover, $||f_{\delta }-L_{T}f_{\delta
}||_{L^{1}} $ can also be estimated quite sharply with some computation.
Indeed, let $L_{T,\xi }$ be a Ulam discretization of $L_{T}$ on a grid of
size $\xi <\delta .$ Remark that 
\begin{equation*}
||f_{\delta }-L_{T}f_{\delta }||_{L^{1}}\leq ||f_{\delta }-L_{T,\xi
}f_{\delta }||_{L^{1}}+||L_{T,\xi }f_{\delta }-L_{T}f_{\delta }||_{L^{1}}
\end{equation*}%
here $||f_{\delta }-L_{T,\xi }f_{\delta }||_{L^{1}}$ can be estimated
explicitly by computation. On the other hand, by Lemma \ref{lemp} \footnote{
Here, since we are in dimension one, we have the freedom to chose $\xi $
very small to minimize this part of the error without increasing too much
the computation time.} 
\begin{equation*}
||L_{T,\xi }f_{\delta }-L_{T}f_{\delta }||_{L^{1}}\leq \xi (2\lambda
_{1}+1)||f_{\delta }||_{BV}+\xi B'||f_{\delta }||_{L^1}.
\end{equation*}%
Here $B'$ is the second coefficient of the Lasota Yorke inequality satisfied
by the one dimensional map $T$ (see Section \ref{LY}) and $2\lambda _{1}$ is
the first coefficient.
\end{remark}

Before the proof we state a Lemma we will use in the following.

\begin{remark}
If $G_{i}$ is family of $\lambda $-contractions, $\mu ,\nu $ probability
measures on the interval, $\mathcal{I}=\{I_{i}\}$ a partition whose diameter
is $\delta $, and $\Gamma _{i}(\mu )=\mathbf{E}(G_{i}(\mathbf{E}(\mu |%
\mathcal{I}))|\mathcal{I})$, then%
\begin{align*}
W(\Gamma _{1}\circ ...\circ \Gamma _{n}(\mu ),\Gamma _{1}\circ ...\circ
\Gamma _{n}(\nu ))& \leq 2\delta +2\lambda \delta +2\lambda ^{2}\delta
+...2\lambda ^{n-1}\delta \\
& \quad +\lambda ^{n}(W(\mu ,\nu )+2\delta ) \\
& \leq \lambda ^{n}(W(\mu ,\nu ))+\frac{2\delta }{1-\lambda }.
\end{align*}
\end{remark}

\begin{lemma}
\label{duo}Let $F:\Sigma \rightarrow \Sigma $ as above, $\mu ,\nu \in
PM(\Sigma )$ with absolutely continuous marginals $\mu _{x},\nu _{x}$. Let
us define $L_{|\delta }=P_{|}LP_{|}$, then 
\begin{equation*}
||L_{|\delta }^{n}\mu -L_{|\delta }^{n}\nu ||\leq \lambda ^{n}+\frac{2\delta
^{^{\prime }}}{1-\lambda }+V(\mu _{_{x}},\nu _{_{x}}).
\end{equation*}
\end{lemma}

\begin{proof}
The proof is similar to the one of Theorem \ref{uno}. Let us consider $%
\{I_{i}\}_{i=1,..,m}$ the intervals where the branches of $T^{n}$ are
defined. Let us consider $\overline{\varphi }_{i}=1_{I_{i}\times I}$ and $%
\mu _{i}=\overline{\varphi }_{i}\mu $, $\nu _{i}=\overline{\varphi }_{i}\nu $
, then $\mu =\sum \mu _{i}$, $\nu =\sum \nu _{i}$ , and then 
\begin{equation}
W(L_{|\delta }^{n}(\mu ),L_{|\delta }^{n}(\nu ))\leq \sum W(L_{|\delta
}^{n}(\mu _{i}),L_{|\delta }^{n}(\nu _{i})).  \label{summ}
\end{equation}

Let us denote $T_{i}=T^{n}|_{I_{i}}$ , as before. Recall that $L_{|\delta }$
and $L_{F}$ have the same behavior on $x$ marginals:$(L_{|\delta }(\mu
))_{x}=(L_{F}(\mu ))_{x}=L_{T}(\mu _{x})$.%
\begin{gather*}
W(L_{|\delta }^{n}(\mu _{i}),L_{|\delta }^{n}(\nu _{i}))\leq
\int_{I}W((L_{|\delta }^{n}\mu _{i})_{\gamma },(L_{|\delta }^{n}\nu
_{i})_{\gamma })dL_{T^{n}}((\nu _{i})_{x})+||L_{T}^{n}((\mu
_{i})_{x})-L_{T}^{n}((\nu _{i})_{x})||_{L^1} \\
\leq \int_{I}W(L_{|\delta }^{n}((\mu _{i})_{T_{i}^{-1}(\gamma )}),L_{|\delta
}^{n}((\nu _{i})_{T_{i}^{-1}(\gamma )}))dL_{T^{n}}((\nu _{i})_{x})+||(\mu
_{i})_{x}-(\nu _{i})_{x}||_{L^1}
\end{gather*}

by the above Remark this is bounded by%
\begin{gather*}
\leq \frac{2\mu _{x}(I_{i})\delta ^{\prime }}{1-\lambda }+\lambda
^{n}\int_{I}W((\mu _{i})_{T_{i}^{-1}(\gamma )},(\nu _{i})_{T_{i}^{-1}(\gamma
)})dL_{T^{n}}((\nu _{i})_{x})+||(\mu _{i})_{x}-(\nu _{i})_{x}||_{L^1} \\
\leq \frac{2\mu _{x}(I_{i})\delta ^{\prime }}{1-\lambda }+\lambda
^{n}\int_{I_{i}}W_{1}((\mu _{i})_{\gamma },(\nu _{i})_{\gamma })d((\nu
_{i})_{x})+||(\mu _{i})_{x}-(\nu _{i})_{x}||_{L^1}
\end{gather*}

where the last step is by change of variable. Hence by Equation \ref{summ}%
\begin{equation}
W(L_{|\delta }^{n}(\mu ),L_{|\delta }^{n}(\nu ))\leq \frac{2\delta ^{\prime }%
}{1-\lambda }+\lambda ^{n}\int_{I}W(\mu _{\gamma },\nu _{\gamma })d(\nu
_{x})+||(\mu )_{x}-(\nu )_{x}||_{L^1}.
\end{equation}
\end{proof}

\begin{proof}[Proof of Proposition \protect\ref{2}]
We recall that in the following, we will consider probability measures
having absolutely continuous marginals. Remark that%
\begin{align*}
W(L^{n}\mu _{0},L_{\delta }^{n}\mu _{0})&\leq W(L^{n}\mu _{0},L_{|\delta
}^{n}\mu _{0})+W(L_{|\delta }^{n}\mu _{0},L_{\delta }^{n}\mu _{0}) \\
&=||L^{n}\mu _{0}-L_{|\delta }^{n}\mu _{0}||+||L_{|\delta }^{n}\mu
_{0}-L_{\delta }^{n}\mu _{0}||.
\end{align*}%
The two summands will be estimated separately in the following items:

\begin{enumerate}
\item Remark that $||L^{n}\mu _{0}-L_{|\delta }^{n}\mu
_{0}||=||\sum_{1}^{n}L^{n-k}(L-L_{|\delta })L_{|\delta }^{k-1}\mu _{0}||.$

Let us estimate $||L^{n-k}(L-L_{|\delta })L_{|\delta }^{k-1}\mu _{0}||$.
Denoting $L_{|\delta }^{k-1}\mu _{0}=g_{k}$, we have $||((L-L_{|\delta
})g_{k})_{\gamma }||\leq 2\delta ^{\prime }$on each leaf $\gamma $ (because $%
||(P_{|}g_{k}-g_{k})_{\gamma }||\leq \delta ^{\prime }$ on each leaf, and $L$
applied to two disintegrated measures having the same marginal does not
increase distance of the respective measures induced on the leaves). Since
the projections $\pi _{1}(Lg_{k})=\pi _{1}(L_{|\delta }g_{k})$ are the same,
by Proposition \ref{prod}, $||L^{n-k}(L-L_{|\delta })g_{k}||\leq \lambda
^{n-k}2\delta ^{\prime }$ where $\lambda $ is the rate of contraction of
fibers. By this $||L^{n}\mu _{0}-L_{|\delta }^{n}\mu _{0}||\leq \frac{%
2\delta ^{\prime }}{1-\lambda }$.

\item In the same way as before $||L_{|\delta }^{n}\mu _{0}-L_{\delta
}^{n}\mu _{0}||=||\sum_{1}^{n}L_{|\delta }^{n-k}(L_{|\delta }-L_{\delta
})L_{\delta }^{k-1}\mu _{0}||.$

Let us estimate $||L_{|\delta }^{n-k}(L_{|\delta }-L_{\delta })L_{\delta
}^{k-1}\mu _{0}||$. Denoting $L_{\delta }^{k-1}\mu _{0}=f_{k}$, we have 
\begin{eqnarray*}
||L_{|\delta }^{n-k}(L_{|\delta }-L_{\delta })f_{k}|| &\leq &||L_{|\delta
}^{n-k}(P_{|}LP_{|}-P_{|}LP_{_{-}}P_{|})f_{k}||+ \\
&&||L_{|\delta
}^{n-k}(P_{|}LP_{_{-}}P_{|}-P_{_{-}}P_{|}LP_{_{-}}P_{|})f_{k}||
\end{eqnarray*}
\end{enumerate}

Let us \ estimate $||L_{|\delta
}^{n-k}(P_{_{-}}P_{|}LP_{_{-}}P_{|}-P_{|}LP_{_{-}}P_{|})f_{k}||$. First let
us consider $k=n$. In this case by transporting horizontally the measure to
average inside each rectangle (recall that $f_{k}$ is a probability measure) 
\begin{equation*}
||(P_{_{-}}P_{|}LP_{_{-}}P_{|}-P_{|}LP_{_{-}}P_{|})f_{k}||\leq \delta .
\end{equation*}

Now let us face the case where $|k-n|\neq 0$. We will give two estimations
for $||L_{|\delta
}^{n-k}(P_{_{-}}P_{|}LP_{_{-}}P_{|}-P_{|}LP_{_{-}}P_{|})f_{k}||$; one will
be suited when $n-k$ is small, and the other when it is large. Then we can
take the minimum of the two estimations. The first estimation is based on
splitting the space into two subsets, in the first subset the map is not too
much expansive, the second set has small measure. This allow to estimate the
maximal expansion rate of $L_{|\delta }^{n-k}$ with respect to the \
Wasserstein distance. The second estimation is based on disintegration,
similar to theorem \ref{uno}.

Now let us face the case where $|k-n|\neq 0$ is small.

We find an estimation for $||L_{|\delta }^{i}(\mu -\nu )||$ for a \ pair of
probability measures $\mu $ and $\nu $ which is suitable when $i$ is small.
Let us divide the space $X$ into two sets $X_{1},X_{2}$ such that: $\mu
(X_{2}),\nu (X_{2})\leq l$, $Lip(F|_{X_{1}})\leq \overline{L}$.%
\begin{eqnarray*}
||L_{|\delta }(\mu -\nu )|| &\leq &2\delta ^{^{\prime }}+\sup_{\substack{ %
Lip(g)\leq 1 \\ ||g||_{\infty }\leq 1}}|\int_{X}gd(LP_{|}\mu -LP_{|}\nu )| \\
&\leq &2\delta ^{^{\prime }}+\sup_{_{\substack{ Lip(g)_{\leq 1} \\ %
||g||_{\infty }\leq 1}}}|\int_{X_{1}}g\circ F~d(P_{|}\mu -P_{|}\nu
)+\int_{X_{2}}g\circ F~d(P_{|}\mu -P_{|}\nu )| \\
&\leq &2\delta ^{^{\prime }}+\overline{L}(||\mu -\nu ||+2\delta ^{\prime
})+2l.
\end{eqnarray*}%
We now iterate, we need that the above general assumptions are preserved.
Recalling that $X_{1}=X-\mathcal{B}\times I,X_{2}=$ $\mathcal{B}\times I$,
we have that, since the map preserves the contracting foliation and since
its one dimensional induced transfer operator $L_{T}$ is a $L^{1}$
contraction then $||L_{T}^{i}f_{\delta }-f_{\delta }||_{L^1}\leq
||L_{T}^{i}f_{\delta }-f_{\delta }-L_{T}^{i}f+f||_{L^1}\leq 2\epsilon ,$ hence 
$L_{|\delta }^{i}\mu _{0}(X_{2})\leq l+2\epsilon ,L_{|\delta }^{i}\overline{%
\mu }(X_{2})\leq l+2\epsilon $ and%
\begin{eqnarray*}
||L_{|\delta }^{n-k}(P_{_{-}}P_{|}LP_{_{-}}P_{|}-P_{|}LP_{_{-}}P_{|})f_{k}||
&\leq &\overline{L}^{n-k}(\delta +2\delta ^{\prime })+\overline{L}%
^{n-k-1}(2\delta ^{^{\prime }}+2l+3\epsilon ) \\
&&+\overline{L}^{n-k-2}(2\delta ^{^{\prime }}+2l+3\epsilon )+...+(2\delta
^{^{\prime }}+2l+3\epsilon ).
\end{eqnarray*}%
Now let us face the case which seems to be suited when $|k-n|\neq 0$ is
large; let us consider%
\begin{equation*}
V((P_{_{-}}P_{|}LP_{_{-}}P_{|}f_{k})_{x}-(P_{|}LP_{_{-}}P_{|}f_{k})_{x}).
\end{equation*}%
Recalling that $L_{T}$ is the one dimensional transfer operator associated
to $T$ and\ $L_{T,\delta }$ is its Ulam discretization with a grid of size $%
\delta $, since $f_{\delta }=(f_{k})_{x}$ is invariant for the one
dimensional approximated transfer operator then%
\begin{equation*}
L_{T,\delta }(\mu )=\pi _{1}(L_{\delta }(\mu \times m))
\end{equation*}%
associated to $L_{\delta }$ then considering the disintegration and the
marginals on the $x$ axis%
\begin{equation*}
||(P_{_{-}}P_{|}LP_{_{-}}P_{|}f_{k})_{x}-(P_{|}LP_{_{-}}P_{|}f_{k})_{x}||_{L^{1}}=||f_{\delta }-L_{T}f_{\delta }||_{L^{1}}.
\end{equation*}%
Thus by Lemma \ref{duo}, 
\begin{equation*}
||L_{|\delta
}^{n-k}(P_{_{-}}P_{|}LP_{_{-}}P_{|}-P_{|}LP_{_{-}}P_{|})f_{k}||\leq \lambda
^{n-k}+\frac{2\delta ^{^{\prime }}}{1-\lambda }+||f_{\delta }-L_{T}f_{\delta
}||_{L^{1}}.
\end{equation*}%
Now, let us estimate $||L_{|\delta
}^{n-k}(P_{|}LP_{|}-P_{|}LP_{_{-}}P_{|})f_{k}||$. We remark that $%
P_{|}f_{k}=P_{_{-}}P_{|}f_{k}$. Indeed $%
f_{k}=(P_{_{-}}P_{|}LP_{_{-}}P_{|})L_{\delta }^{k-2}\mu _{0}$ thus it has
already averaged on the horizontal direction, this is not changed by
applying $P_{|}$, and then applying again $P_{_{-}}$ has no effect. Hence $%
||L_{|\delta }^{n-k}(P_{|}LP_{|}-P_{|}LP_{_{-}}P_{|})f_{k}||=0$.

Summarizing, considering that we can take the minimum of the two different
estimations and putting all small terms in a sum, we have Equation (\ref{err}%
).
\end{proof}

\section{The algorithm}

The considerations made above justify an algorithm for the computation with
explicit bound on the error for the physical invariant measure of Lorenz\
like systems we describe informally below.

\begin{algorithm}
\begin{enumerate}
\item Input $\delta ,\delta ^{\prime }$. Compute a $L^1$ approximation for the
marginal one dimensional invariant measure $f_{\delta }$ of the one
dimensional induced map $T$ (see Section \ref{1d} for the details)

\item Input $n.$ Use Theorem \ref{uno} to estimate $W(L_{F}^{n}(\mu
_{0}),L_{F}^{n}(\overline{\mu })).$

\item Use Proposition \ref{2} to estimate the distance $W(L_{\delta
}^{n}(\mu _{0}),L_{F}^{n}(\mu _{0}))$

\item Compute an approximation $\tilde{\mu}$ for $L_{\delta }^{n}(\mu _{0})$
up to an error $\eta $.

\item Output $\tilde{\mu}$ and $W(L_{F}^{n}(\mu _{0}),L_{F}^{n}(\mu ))+$ $%
W(L_{\delta }^{n}(\mu _{0}),L_{F}^{n}(\mu _{0}))+\eta.$
\end{enumerate}
\end{algorithm}

\begin{proposition}
What is proved above implies that $\tilde{\mu}$ is such that%
\begin{equation*}
W(\tilde{\mu},\overline{\mu })\leq W(L_{F}^{n}(\mu _{0}),L_{F}^{n}(\overline{%
\mu }))+W(L_{\delta }^{n}(\mu _{0}),L_{F}^{n}(\mu _{0}))+\eta .
\end{equation*}
\end{proposition}

Of course this is an a posteriori estimation for the error. Hence it might
be that the error of approximation is not satisfying. In this case one can
restart the algorithm with a larger $n$ and smaller $\delta ,\delta ^{\prime
}$.

\begin{remark}
\label{works}We remark that for each $\varepsilon $, there are  integers $ m, n$
and grid sizes $\delta $, $\delta ^{\prime }$, $\xi $ such that the above
algorithm applied to $F^{m}$ computes a measure $\tilde{\mu}$ such that $W(%
\tilde{\mu},\overline{\mu })\leq \varepsilon $.

Indeed choose $m$ such that $\lambda ^{m}\leq \frac{\varepsilon }{10}$ and $%
n=2$ iterations. Choose $\delta $ such that $||f-f_{\delta }||_{L_{1}}\leq 
\frac{\varepsilon }{10}$ (see e.g. \cite{GN}, Section 5.1 for the proof that
such an approximation is possible up to any small error) then by Theorem \ref%
{uno}, $||L^2 \mu _{0}-\overline{\mu }||\leq \frac{\varepsilon }{5}$.

Let us suppose that $\delta $ and $\xi $ are so small that $\delta +\frac{%
4\delta ^{^{\prime }}}{1-\lambda }+||f_{\delta }-L_{T,\xi }f_{\delta
}||_{L^{1}}+\xi (2\lambda _{1}+1)||f_{\delta }||_{BV}+\xi B'||f_{\delta
}||_{L^1}\leq \frac{\varepsilon }{10}$. This is possible because $||f_{\delta
}-L_{T,\xi }f_{\delta }||_{L^{1}}\leq ||f_{\delta }-f_{\xi
}||_{L^{1}}+||f_{\xi }-L_{T,\xi }f_{\delta }||_{L^{1}}\leq 2||f_{\delta
}-f_{\xi }||_{L^{1}}$ and $||f_{\delta }-f_{\xi }||_{L^{1}}\leq ||f_{\delta
}-f||_{L^{1}}+||f-f_{\xi }||_{L^{1}}.$

Then by Proposition \ref{2}, $||L^2\mu _{0}-L^2_{(\delta ,\delta ^{\prime })}\mu
_{0}||\leq \frac{\varepsilon }{5}$ and we have that $W(\tilde{\mu},\overline{%
\mu })$ can be made as small as wanted.

It is clear that the choice of the parameters which is given above might be
not optimal, and setting a suitable $m$ or $n$ we might achieve a better
approximation. The purpose of this remark is just to show that our method
can in principle approximate the physical measure up to any small error.
\end{remark}

\section{Dimension of Lorenz like attractors}

\label{dim}

We show how to use the computation of the invariant measure to compute the
fractal dimension of a Lorenz like attractor.

We recall and use a result of Steinberger \cite{Stein00} which gives a
relation between the entropy of the system and its geometrical features.

Let us consider a map $F:Q\rightarrow Q$, $F(x,y)=(T(x),G(x,y))$ satisfying
the items 1)...4) in the Introduction, and

\begin{itemize}
\item $F((c_{i},c_{i+1})\times \lbrack 0,1])\cap F((c_{j},c_{j+1})\times
\lbrack 0,1])=\emptyset $ for distinct $i,j$ with $0\leq i,j<N$.
\end{itemize}

%

Let us consider the projection $\pi :Q\rightarrow I$, set $%
V=\{(c_{i},c_{i+1}),1\leq i\leq N\}$, consider $V_{k}=%
\bigvee_{i=0}^{k}T^{-i}V$. For $x\in E$ let $J_{k}(x)$ be the unique element
of $V_{k}$ which contains $x$. We say that $V$ \emph{is a generator} if the
length of the intervals $J_{k}(x)$ tends to zero for $n\rightarrow \infty $
for any given $x$. For a topologically mixing piecewise expanding maps $V$
is a generator. Set 
\begin{equation*}
\psi (x,y)=\log |T^{\prime }(x)|\quad \mbox{and}\quad \varphi (x,y)=-\log
|(\partial G/\partial y)(x,y)|.
\end{equation*}%
The result we shall use to estimate the dimension is the following

\begin{theorem}
\cite[Theorem 1]{Stein00} \label{th:Steinberg} Let $F$ be a two-dimensional
map as above and $\mu _{F}$ an ergodic, $F$-invariant probability measure on 
$Q$ with the entropy $h_{\mu }(F)>0$. Suppose $V$ is a generator, $\int
\varphi ~d\mu _{F}<\infty $ and $0<\int \psi ~d\mu _{F}<\infty $. If the
maps $y\mapsto \varphi (x,y)$ are uniformly equicontinuous for $x\in
I\setminus \{0\}$ and $1/|T^{\prime }|$ has finite universal $p$- Bounded
Variation, then 
\begin{equation*}
d_{\mu }(x,y)=h_{\mu }(F)\left( \frac{1}{\int \psi ~d\mu }+\frac{1}{\int
\varphi ~d\mu }\right)
\end{equation*}%
for $\mu $-almost all $(x,y)\in Q$.
\end{theorem}

\begin{remark}
We remark that since the right hand of the equation does not depend on $%
(x,y) $, this implies that the system is exact dimensional.

We also remark that $\int \psi ~d\mu $ \ can be computed by the knowledge of
the measure of the 1 dimensional map under small errors in the $L^{1}$ norm
\ and having a bound for its density (see Section \ref{8}).
\end{remark}

The following should be more or less well known to the experts, however
since we do not find a reference we present a rapid sketch of proof.

\begin{lemma}
If $(F,\mu )$ as above is a computable dynamical system \footnote{%
For the precise definition, see \cite{GHR07}. In practice, since the
invariant measure is computable starting from the definition of $F$, this is
satisfied by Remark \ref{works} for example when $F$ is given explicitly
like in Equations \ref{eq:lorenz} and \ref{eq:lorenz_y}.} then 
\begin{equation*}
h_{\mu }(F)=h_{\mu _{x}}(T).
\end{equation*}
\end{lemma}

\begin{proof}
(sketch) We will use the equivalence between entropy and orbit complexity in
computable systems (\cite{GHR07}). Since $h_{\mu }(F)\geq h_{\mu _{x}}(T)$
is trivial, we only have to prove the opposite inequality. What we are going
to do is to show that from an approximate orbit for $T$ and a finite
quantity of information, one can recover (recursively) an approximated orbit
for $F$.

We claim that, for most initial conditions $x$, starting from an $r$
approximation $p_{1},\ldots ,p_{n}\in \mathbb{Q}$ for the $T$ orbit of $\pi
_{1}(x)$ (by $r$ approximation we mean that $T^{i}(\pi _{1}(x))\in
B(p_{i},r) $, we recall that we take the sup norm on $\mathbb{R}^{2}$) we
can recover a $K$ approximation $x_{1},...,x_{n}$ for the orbit of $x$ by $F$
(hence $F^{i}(x)\in B(x_{i},K)$) for some $K$ not depending on $n$. Let us
denote the rectangle with edges $r,r^{\prime }$and center $x$ by $%
B(r,r^{\prime },x) $. Let us consider 
\begin{equation*}
C=\underset{x_{1},x_{2}\in B(p_{i},r)}{\sup }\frac{|G(x_{1},y)-G(x_{2},y)|}{%
|x_{1}-x_{2}|}.
\end{equation*}%
By Item \ref{it:item_2}, this is bounded.

Let us describe how to find the sequence $x_{i}$ by $p_{i}$ inductively.
Suppose we have found $x_{i}$, such that $\pi _{1}x_{i}=p_{i}$. Let us
suppose $r$ is so small that $\lambda r^{\prime }+Cr\leq r^{\prime }$. Let $%
K=\max (r,\lambda r^{\prime }+Cr)$; by the contraction in the vertical
direction%
\begin{equation*}
F(B(r^{\prime },r,x_{i}))\cap \pi _{1}^{-1}(B(p_{i+1},r))\subseteq B(\lambda
r^{\prime }+Cr,r,x_{i+1})\subset B(x_{i+1},K).
\end{equation*}%
for some $x_{i+1}$ such that $\pi _{1}x_{i+1}=p_{i+1}$. And if $F$ is
computable, such an $x_{i+1}$ can be computed by the knowledge of $x_{i}$, $%
p_{i}$, $F$, $r$, $r^{\prime }$.

Remark that if $r$ is as above then $\ F(B(r^{\prime },r,x_{i}))\cap \pi
_{1}^{-1}(B(p_{i+1},r))\subseteq B(r,r^{\prime },x_{i+1})$ and we can
continue the process. Hence by the computability of the map, knowing $x$ at
a precision $r$ (to start the process) and $p_{1},\ldots ,p_{n}$ \ we can
recover suitable $x_{1},\ldots ,x_{n}$ at a precision $K$ (by some
algorithm, up to any accuracy).

This shows that from an encoding of $p_{1},\ldots,p_{n}$ and a fixed
quantity of information, one can recover a description of the orbit of $x$
at a precision $K$. By this the orbit complexity of typical orbits in $%
(I^{2},F)$ is less or equal than the one in $(I,T)$ and if everything is
computable, these are equal to the respective entropies (see \cite{GHR07}).
Thus $h_{\mu }(F)\leq h_{\mu _{x}}(T)$.
\end{proof}

\begin{remark}
By the above lemma and $h_{\mu }(F)=\int \psi ~d\mu $ then it follows that:%
\begin{equation*}
d_{\mu }(x,y)=1+\frac{\int \psi ~d\mu }{\int \varphi ~d\mu }.
\end{equation*}
\end{remark}

\section{Implementation of the algorithms \label{8}}

Here we briefly discuss some numerical issue arising in the implementation
of the algorithms and the choices we made to optimize it.

\subsection{Reducing the number of elements of the partition}

\label{subsec:rid}

Our goal is to compute a Ulam like approximation of the $2$-dimensional map.
Since, as noticed in the introduction, the complexity of the problem and the
number of cells involved in the discretization, grows quadratically and hence too fast if we
consider the whole square. The idea is to restrict the dynamics to a
suitable invariant set containing the attractor.

We remark that, since the image of the first iteration of the map $F(Q)$ is
again invariant for the dynamics, we can restrict to the dynamics on some
suitable set containing it (i.e. to the elements of the partition that
intersect the image of the map) and compute the Ulam approximation of this
restricted map. As a matter of fact, we could take an higher iteration of
the map and narrow even more the size of the chosen starting region.

Therefore we have to find rigorously a subset of the indexes, such that the
union of the elements of the partition with indexes in this subset contains
the attractor. To do so we use the containment property of interval
arithmetics; if $\tilde{F}$ is the interval extension of $F$ and $R$ is a
rectangle in the continuity domain, then $\tilde{F}(R)$ is a rectangle such
that $F(R)\subset\tilde{F}(R)$.

We want to compute rigorously a subset that contains the image of the map;
at the same time we would like this set to be ``small''. We divide each of the
continuity domains along the $x$ in $k$ homogeneous pieces and, at the same
time we partition homogeneously in the $y$ direction, constructing a
partition $\mathcal{P}$, whose elements we denote by $R$, which is coarser 
than the Ulam partition. 

To compute the indexes of the elements of the Ulam partition that intersect
$F(Q)$, for each rectangle $R\in\mathcal{P}$ we take $\tilde{F}(R)$ and mark the indexes
whose intersection with $\tilde{F}(R)$ is non empty.
Doing so we obtain a subset of the indexes which is guaranteed to contain the image of
the map and, therefore, the attractor. 

This reduces dramatically the size of
the problem; in our example with a size of $16384\times 1024$ this permits
us to reduce the number of coefficients involved in the computation from $2^{24}=16384\times 1024 $ to $351198$.

\subsection{Computing the Ulam matrix}

To compute with a given precision the coefficients of the Ulam approximation we have to compute $%
P_{ij}:=m(F^{-1}R_{j}\cap R_{i})/m(R_{i})$; what we do is to find a
piecewise linear approximation for the preimage of $R_{j}$. We remark that
the preimages of the vertical sides of $R_{j}$ are vertical lines, due to
the fact that the map preserves the vertical foliation, while the preimage
of the horizontal sides are lines which are graphs of functions $\phi (x)$.

We approximate $F^{-1}R_{j}$ by a polygon $\tilde{P_{j}}$ and compute the
area of the intersection between the polygon and $R_{i}$ with a precribed
bound on the error.

We present a drawing to illustrate our ideas: figure \ref{fig:computing}. In
the figure the intersection of the rectangle and the preimage, represented by the darkest region, 
 is the region whose area we want to compute.

Denote by $R_{j,l}$ and $R_{j,u}$ the quotes of the lower ad upper
sides of $R_j$; inside a continuity domain we can apply the implicit
function theorem and we know that there exists $\phi_l$ and $\phi_u$ such
that 
\begin{equation*}
G(x,\phi_l(x))=R_{j,l}\quad G(x,\phi_u(x))=R_{j,u}.
\end{equation*}
Computing $P_{ij}$ is nothing else that computing the area of the
intersection between $R_i$ and the area between the graphs of $\phi_u$ and $%
\phi_l$, i.e., computing rigorously the difference between the integrals of
the two functions 
\begin{equation*}
\chi_u=\min\{\phi_u,R_{i,u}\}, \quad \chi_l=\max\{\phi_l,R_{i,l}\},
\end{equation*}
over the interval $R_i\cap T^{-1}(\pi_x(R_j))$, where $T$ is the one
dimensional map and $\pi_x$ is the projection on the $x$ coordinate.

We explain some of the ideas involved in the computation of the integral of $%
\chi_u$; the procedure and the errors for $\chi_l$ follow from the same
reasonings.

The main idea consists in approximating $\phi_u$ by a polygonal and estimate
the error made in computing the area below its graph and its intersection
point with the quote $R_{i,u}$.

In figure \ref{fig:computing}, $\phi _{l}$ is the preimage of the lower side
of $R_{j}$, $\phi _{u}$ is the preimage of the upper side and $\tilde{\phi}%
_{u}$ is the approximated preimage of the upper side (with four vertices).

From straightforward estimates it is possible to see that the error made
using the polygonal approximation $\tilde{\phi}_u$ when we compute the
integral below the graph of $\phi_u$ depends on the second derivative of $%
\phi_u$, while the error made on computing the intersection between $\phi_u$
and $R_{i,u}$ depends on the distortion of $\phi_u$.

Please note that $\phi _{u}^{\prime \prime }$ and the distortion of $\phi
_{u}$ go to infinity near the discontinuity lines; at the same time, in our
example the contraction along the $y$-direction is strongest near the
discontinuity lines. Therefore, if the discretization is fine enough, \emph{%
for rectangles near the discontinuity lines}, the image of $R_{i}$ is
strictly contained between two quotes; this implies that in this specific
case 
\begin{equation*}
\chi _{u}\equiv R_{i,u}\quad \chi _{l}\equiv R_{i,l},
\end{equation*}%
therefore 
\begin{align*}
P_{ij}&=\frac{m(F^{-1}R_{j}\cap R_{i})}{m(R_{i})}=\frac{\int_{T^{-1}(%
\pi_1(R_j))}\chi_u-\chi_l dm}{m(R_{i})}=\frac{\int_{T^{-1}(\pi_1(R_j))}
1/\delta^{\prime }dm}{m(R_{i})} \\
&=\frac{\int_{T^{-1}(\pi_1(R_j))} 1 dm}{m(\pi_1^{-1}R_{i})}=\frac{%
m(T^{-1}(\pi_1(R_j))\cap \pi_1(R_i))}{m(\pi_1(R_j))},
\end{align*}
i.e. in these particular cases, the coefficients depend only on what happens
along the $x$-direction.

\begin{figure}[tbp]
\begin{center}
\def\svgwidth{80mm}
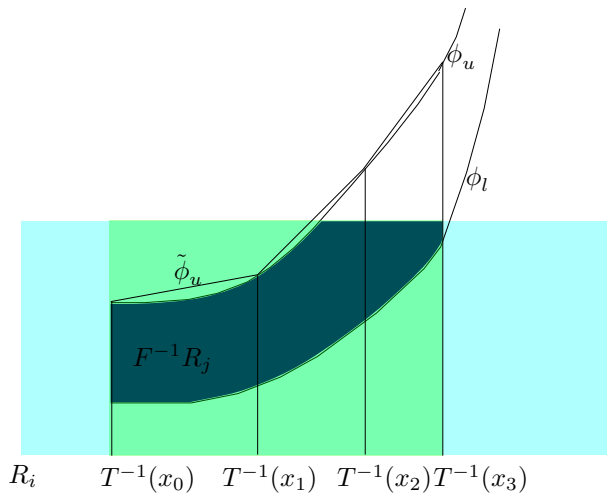
\end{center}
\caption{The rectangle $R_i$, the preimage $F^{-1} R_j$ with its upper
boundary $\protect\phi_u$ and its lower boundary $\protect\phi_l$ and the
linearization of the upper boundary $\tilde{\protect\phi}_u$.}
\label{fig:computing}
\end{figure}

\section{Numerical experiments}

In this section we show the results of a rigorous computation on a Lorenz-like map $F$ . 
The C++ codes and the numerical data are available at:  

\centerline{http://im.ufrj.br/~nisoli/CompInvMeasLor2D}

\subsection{Our example}

\label{sec:victim} In our experiments we analize the fourth iterate (i.e. $%
F^{4}:Q\rightarrow Q$) of the following two dimensional Lorenz map 
\begin{equation*}
F(x,y)=(T(x),G(x,y))
\end{equation*}%
with 
\begin{equation}
T(x)=\left\{ 
\begin{array}{cc}
\theta |x-1/2|^{\alpha } & 0\leq x<1/2 \\ 
1-\theta |x-1/2|^{\alpha } & \frac{1}{2}<x\leq 1%
\end{array}%
\right.  \label{eq:lorenz}
\end{equation}%
with constants $\alpha =51/64$, $\theta =109/64$, and 
\begin{equation}
G(x,y)=\left\{ 
\begin{array}{cc}
(y-1/2)|x-1/2|^{\beta }+1/4 & 0\leq x<1/2 \\ 
(y-1/2)|x-1/2|^{\beta }+3/4 & \frac{1}{2}<x\leq 1%
\end{array}%
\right.  \label{eq:lorenz_y}
\end{equation}%
with $\beta =396/256$.

The graph of $T$ is plotted in figure \ref{fig:lorenz_func}. In subsection %
\ref{sec:Lorenz_1D} we give the results of the rigorous computation of the
density plotted in figure \ref{fig:lorenz_dens}.

\begin{figure}[!h]
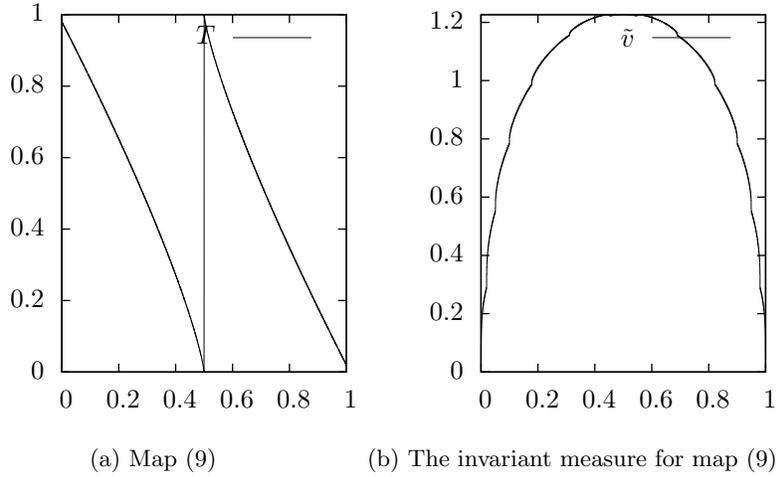

\begin{subfigure}[b]{0.45\textwidth}
      \input{./lorenz_func.tex}
      \caption{Map \eqref{eq:lorenz}}
      \label{fig:lorenz_func}
  \end{subfigure}  
\begin{subfigure}[b]{0.45\textwidth}
    \input{./lorenz_dens.tex}
    \caption{The invariant measure for map \eqref{eq:lorenz}}
    \label{fig:lorenz_dens}
  \end{subfigure}
\caption{The one dimensional map: \eqref{eq:lorenz}.}
\end{figure}

\subsection{The Lorenz $1$-dimensional map}

\label{sec:Lorenz_1D}

The first step in our algorithm is the approximation of the invariant
measure of the one dimensional induced map. The algoritm we use is the one
described in \cite{GN} with the estimations described in Section \ref{1d}.

We consider the fourth iterate $T^4$ of  the Lorenz $1$-dimensional map $T$ given by Equation \ref{eq:lorenz} with $\alpha=51/64$ and $\theta=109/64$ and estimate the coefficients of its Lasota Yorke inequality (see Section \ref{1d} ). We have that 
\begin{equation*}
\frac{2}{\min (d_{i}-d_{i+1})}\leq 37.8247.
\end{equation*}

Moreover, $\sup |1/T^{\prime }|\leq 0.16$, we fix $l=30$ obtaining that 
\begin{equation*}
\frac{1}{2}\int_{I_{l}}|\frac{T^{\prime \prime }}{(T^{\prime })^{2}}|\leq
0.46
\end{equation*}
and that $\lambda_1  \leq 0.763$.

We have therefore that the second coefficient of the Lasota Yorke inequality is  $B\leq 285.053$.
From the experiments, on a partition of $1048576$ elements, with a matrix such that
each component was computed with an error of $2^{-43}$ we have that the number of iterations needed to contract the zero average space is  $N=8$.
  
Therefore the rigorous error on the computation of the one dimensional
measure is:\footnote{The additional parameters which are involved in this computation (see \cite{GN} for the meaning) are $B^{\prime }\leq 67.83$,  $N_{\varepsilon}=7$,  and the number of iterates necessary to contract
the unit simplex to a diameter of $0.0001$ was $10$ (i.e., the numerical accuracy
with which we know the eigenvector for the matrix).
Therefore the rigorous error  is estimated as: 
\begin{equation*}
||f-\tilde{v}||_{L^1}\leq \frac{2\cdot 8\cdot 285.053}{1048576}+2\cdot 7\cdot
1048576\cdot 2^{-43}+0.0001\leq 0.005.
\end{equation*}
}
  
\begin{equation*}
||f-\tilde{v}||_{L^1}\leq 0.005.
\end{equation*}

\subsection{Estimating the measure for the Lorenz $2$-dimensional map}

\label{meas2d} In our numerical experiments we used a partition of the
domain of size $\delta =1/16384=2^{-14}$ elements in the $x$ direction and
of $\delta ^{\prime }=2^{-10}$ in the $y$-direction and reduced the number of elements we consider as explained in
\ref{subsec:rid}. 

We computed the Ulam discretization of the fourth iterate of the Lorenz $2$%
-dimensional map; using the method explained in subsection \ref{subsec:rid}
our program needed to compute approximatively $351198$ cofficients of the
matrix.

As explained in Subsection \ref{sec:Lorenz_1D} the one dimensional map
satisfies a Lasota Yorke  inequality with coefficients $\lambda _{1}\leq 0.763$, $B\leq
285.53$ and we have a computed approximated density on a partition of size $%
\xi=1/1048576=2^{-20}$ such that $||f-f_{\xi}||_{L^1}\leq 0.005$.

To estimate $||f_{\delta}||_{BV},$ as required in Proposition \ref{2} we use the
upper bound (Lemma \ref{Lemma:BV}) 
\begin{equation*}
||f_{\delta}||_{BV}\leq \text{Var}(f_{\delta})+2||f_{\delta}||_{L^1},
\end{equation*}%
which gave, constructing $f_{\delta}$ by averaging $f_{\xi}$ on the coarser
partition that 
\begin{equation*}
||f_{\delta}||_{BV}<4.37,\quad ||f_{\xi}||_{BV}<4.46,\quad
||f_{\xi}-f_{\delta}||_{L^1}\leq \frac{4.46}{16384}\leq 0.0003.
\end{equation*}

To apply Proposition \ref{2} and Remark \ref{r6}  we compute that in our example $\lambda \leq 0.014$, and we chose to take intervals of size $ 2/1048576$ near the discontinuity points to estimate $l$ and $\bar{L}$,
obtaining respectively that $l<3.2\cdot 10^{-5}$ and $\bar{L}<1277$.

Since 
\begin{align*}
||L_{T}f_{\delta }-f_{\delta }||_{L^1}\leq ||L_{T}f_{\delta }-L_{T,\xi
}f_{\delta }||_{L^1}+||L_{T,\xi }f_{\delta }-f_{\delta }||_{L^1},
\end{align*}%
then 
\begin{equation*}
||L_{T}f_{\delta }-f_{\delta }||_{L^1}\leq 0.00034.
\end{equation*}

Let $\mu_0$ be as defined in \eqref{eq:mu_0}. We apply Theorem \ref{uno} to
estimate the distance $W(\overline{\mu },L_{F}^{3}\mu _{0})$ after $3$
iterations, since $\lambda ^{3}\leq 3\cdot 10^{-6}$.

Let us consider the error estimate; let $\overline{\mu }$, with marginal $f$%
, be the physical invariant measure for $F$: 
\begin{equation*}
W(\overline{\mu },L_{F}^{3}\mu _{0})=W(L_{F}^{3}\overline{\mu },L_{F}^{3}\mu
_{0})\leq ||f-f_{\delta}||_{L^1}+\lambda ^{3}\leq
||f-f_{\xi}||_{L^1}+||f_{\xi}-f_{\delta}||_{L^1}+\lambda ^{3};
\end{equation*}%
therefore: 
\begin{equation*}
W(\overline{\mu },L_{F}^{3}\mu _{0})\leq 0.0054.
\end{equation*}

Now, we need to take into account the fact that we are iterating an
approximated operator; referring to Proposition \ref{2} and looking at the data of
our problem, we can see that, when taking the minimum, is always the second
member that is chosen. Then, the explicit formula is: 
\begin{align*}
W(L_{F}^{n}\mu_0 ,L_{\delta }^{n}\mu_0 )
\leq n\cdot \frac{2\delta ^{\prime }}{1-\lambda }+\delta+\sum_{i=1}^{n-1}%
\bigg(\lambda ^{n-i}+||L_{T}f_{\delta}-f_{\delta}||_{L^1}\bigg).
\end{align*}

Therefore 
\begin{align*}
W(L_F^3 \mu_0,L^3_{\delta}\mu_0)&\leq 3\cdot \frac{2\delta^{\prime }}{%
1-\lambda}+\delta+2\cdot ||L_T f_{\delta}-f_{\delta}||_{L^1}+\sum_{i=1}^2
\lambda^{3-i} \\
&\leq 3\cdot 0.002+0.00007+0.00068+0.015\leq 0.022
\end{align*}
and 
\begin{equation*}
W(\overline{\mu},L^3_{\delta} \mu_0)\leq 0.028.
\end{equation*}

In figure \ref{fig:dentity_2d} we present an image of the computed density,
on the partition $16384\times 1024$.

\begin{figure}[tbp]
\begin{center}
\includegraphics[width=80mm]{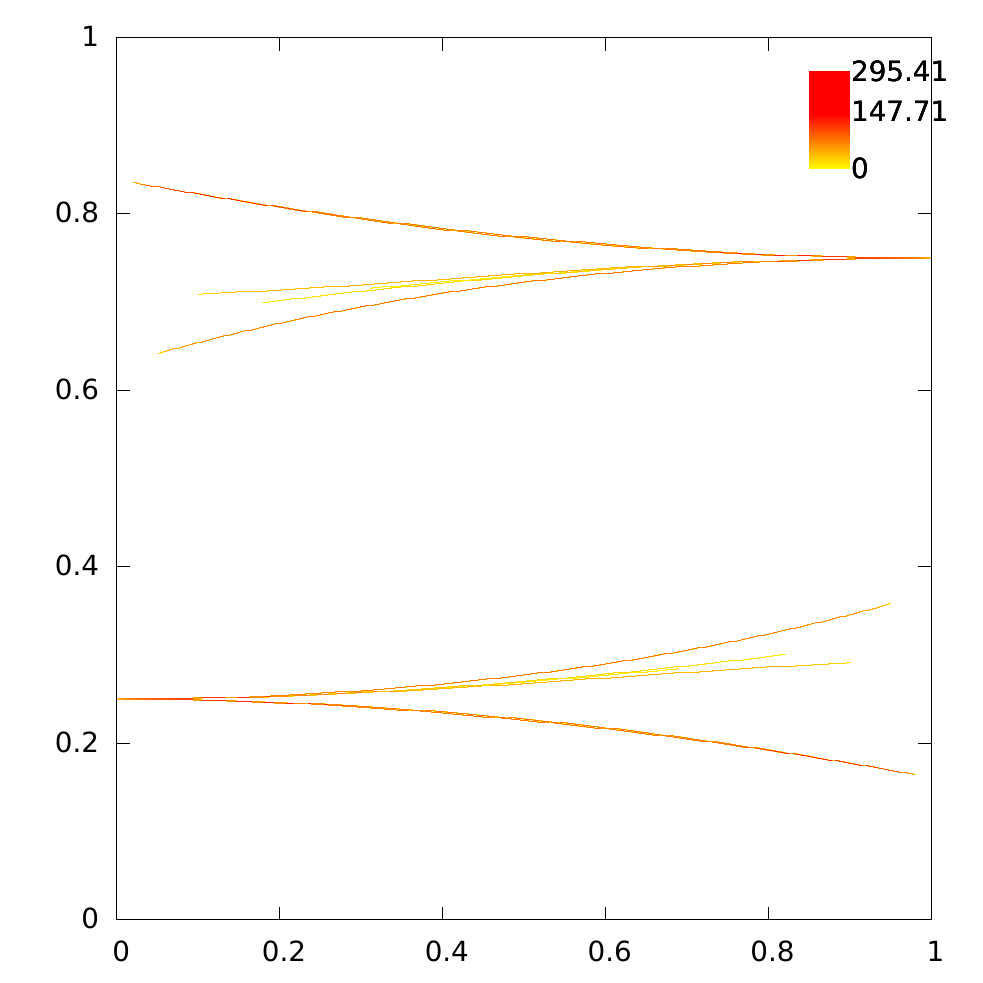}
\end{center}
\caption{Density on a grid $16384\times 1024$}
\label{fig:dentity_2d}
\end{figure}

\section{Estimating the dimension}

\label{dimension} Here we use the results explained in Section \ref{dim} to
rigorously approximate the dimension of the above computed invariant measure.

Inspecting \eqref{eq:lorenz_y}, it is possible to see that in our case we
have that $\partial_y G(x,y)$ is constant along the fibers. More explictly,
by \eqref{eq:lorenz},\eqref{eq:lorenz_y} we have that: 
\begin{equation*}
\log(|\partial_x T(x)|)=\log(\theta)+\log(\alpha)+(\alpha-1)\log |x-1/2|,
\end{equation*}
\begin{equation*}
\quad \log(|\partial_y G(x,y)|)=\beta\log|x-1/2|
\end{equation*}
Therefore, if $\mu$ is the invariant measure for the Lorenz $2$-dimensional
map, to estimate the dimension, we have to estimate 
\begin{equation*}
\int_0^1 \log |x-1/2| d\mu_x,
\end{equation*}
where $d\mu_x$ has density $f$.

On one side, the function $\log |x-1/2|$ is unbounded, on the other side, we
only know an approximation of the density, that we denote by $f_{\delta}$.
Let us estimate from above and from below of the integral.

To give the estimate from above, we take a small interval $%
(1/2-\epsilon_1,1/2+\epsilon_1)$ and we define a new function 
\begin{equation}
\tilde{\psi_1}(x)=\left\{ 
\begin{array}{cc}
\log|x-1/2| & x \in [0,1]\setminus (1/2-\epsilon_1,1/2+\epsilon_1) \\ 
\log|\epsilon_1| & x\in (1/2-\epsilon_1,1/2+\epsilon_1)%
\end{array}%
\right.
\end{equation}

Therefore we have: 
\begin{equation*}
\int_0^1 \log|x-1/2|df\leq\int_0^1 \tilde{\psi_1} df_{\delta}+|\log(%
\epsilon_1)| \cdot ||f-f_{\delta}||_{L^1}.
\end{equation*}

Now, we want to estimate the integral from below; the idea is again to split
the integral in two parts. By the Lasota-Yorke inequality we know that the
BV norm of $f$ is limited from above by  the second coefficient  of the Lasota Yorke inequality $B$; therefore we have that $%
||f||_{\infty}\leq B$.

Again, we take a small interval $(1/2-\epsilon_2,1/2+\epsilon_2)$. We have
that 
\begin{equation*}
\int_{1/2-\epsilon_2}^{1/2}\log|x-1/2|df\geq B
\int_{1/2-\epsilon_2}^{1/2}\log|x-1/2|dx,
\end{equation*}
where $dx$ is the Lebesgue measure on the interval $[0,1]$. Therefore 
\begin{equation*}
\int_{1/2-\epsilon_2}^{1/2}\log|x-1/2|df\geq B \epsilon_2
(\log(\epsilon_2)-1).
\end{equation*}
Let 
\begin{equation}
\tilde{\psi_2}(x)=\left\{ 
\begin{array}{cc}
\log|x-1/2| & x \in [0,1]\setminus (1/2-\epsilon_2,1/2+\epsilon_2) \\ 
0 & x\in (1/2-\epsilon_2,1/2+\epsilon_2)%
\end{array}%
\right.
\end{equation}
Then we have that: 
\begin{equation*}
\int_0^1 \log |x-1/2|  df\geq \int_0^1 \psi_2
df_{\delta}-|\log(\epsilon_2)|\cdot ||f-f_{\delta}||_{L^1}-2 B \epsilon_2
|\log(\epsilon_2)-1|.
\end{equation*}

Using the computed invariant measure we have the following proposition.

\begin{theorem}
The dimension of the physical invariant measure for the map described in
Section \ref{sec:victim} lies in the interval $[1.24063,1.24129]$.
\end{theorem}

\begin{remark}
The high number of significative digits depends on the fact that, due to the properties of the chosen map, in this
estimate we are using only the one dimensional approximation of the measure $f_{\delta}$, which we know with high precision.
\end{remark}

\section{A non-rigorous dimension estimate}

As a control, we implemented a non-rigorous computation of the correlation dimension of
the attractor, following the classical approach described in \cite{Ka-Sc}. Let $%
\theta $ be the heavyside function, i.e., $\theta (x)=0$ if $x\leq 0$ and $%
\theta (x)=1$ if $x\geq 0$. Let $x$ be a point on the attractor and $%
x_{i}:=F^{i}(x)$, for $i=1,\ldots n$; define 
\begin{equation*}
C(\varepsilon )=\frac{2}{n(n-1)}\sum_{i=1}^{n}\sum_{j=i+1}^{n}\theta
(\varepsilon -d(x_{i},x_{j})),
\end{equation*}%
where $d(x,y)=\max (|x_{1}-y_{1}|,|x_{2}-y_{2}|)$ is the $\max $ distance;
in the following, we denote by $B_{\varepsilon }(x)$ the ball with respect
to the distance $d(x,y)$.
This permits us to define a non rigorous estimator for the local dimension
of $\mu$, the so called correlation dimension: 
\begin{equation*}
\tilde{d}_{\mu}:=\lim_{\varepsilon\to 0} \frac{\log C(\varepsilon)}{%
-\log(\varepsilon)}.
\end{equation*}

We implemented an algorithm that uses this idea and applied it to a non
rigorous experiment where we fixed a family of tresholds $%
\varepsilon_k=2^{-9-k}$ for $k=0,\ldots,18$ with an orbit (a pseudo orbit) of length $%
n=2097152$, and interpolated the results (in a log-log scale) with least
square methods. The linear coefficient of the interpolating line should be
an approximation of $\tilde{d}_{\mu}$.

The linear coefficient we obtain from our computations is $1.236$ which is
near our rigorous estimate of $[1.24063,1.24129]$.

\section{Appendix: computing the invariant measure of piecewise expanding
maps with infinite derivative\label{1d}}

\paragraph{Approximating fixed points and the invariant measures.}

In this section we see how to estimate the invariant measure of a one
dimensional piecewise expanding map to construct the starting measure for
our iterative method.

The method we used is the one explained in \cite{GN}. In that paper
piecewise expanding maps with finite derivative were considered, while here
the map has infinite derivative. We briefly explain the method and show the
estimation which allows to use it for the infinite derivative case.

In \cite{GN} the computation of invariant measures was faced by a fixed
point stability result. The transfer operator is approximated by a suitable
discretization (as the Ulam one described before) and the distance between
the fixed point of the original operator and the dicretized one is estimated
by the stability statement.

To use it we need some a priori estimation and some computation which is
done by the computer.

Let us introduce the fixed point stability statement we are going to use.

Let us consider a restriction of the transfer operator to an invariant
normed subspace (often a Banach space of regular measures) $\mathcal{%
B\subseteq }SPM(X)$ and let us still denote it by $L$:$\mathcal{B\rightarrow
B}$. Suppose it is possible to approximate $L$ in a suitable way by another
operator $L_{\delta }$ for which we can calculate fixed points and other
properties (as an example, the Ulam discretization with a grid of size $%
\delta $).

It is possible to exploit as much as possible the information contained in $%
L_{\delta }$ to approximate fixed points of $L$. Let us hence suppose that $%
f,$ $f_{\delta }\in \mathcal{B}$ are fixed points, respectively of $L$ and $%
L_{\delta }$.

\begin{theorem}[see \protect\cite{GN}]
\label{gen}\emph{Let } $V=\{\mu \in \mathcal{B}~s.t.~\mu (X)=0\}$\emph{.
Suppose:}

\begin{enumerate}
\item $||L_{\delta }f-Lf||_{\mathcal{B}}\leq \epsilon $\emph{\ }

\item $\exists \,N$\emph{\ such that }$\forall v\in V~,~||L_{\delta
}^{N}v||_{\mathcal{B}}\leq \frac{1}{2}||v||_{\mathcal{B}}$\emph{\ }

\item $L_{\delta }^{i}|_{V}$ \emph{is continuous, let }$C_{i}=\sup_{g\in V}%
\frac{||L_{\delta }^{i}g||_{\mathcal{B}}}{||g||_{\mathcal{B}}}$
\end{enumerate}

\emph{Then }%
\begin{equation}
||f_{\delta }-f||_{\mathcal{B}}\leq 2\epsilon \sum_{i\in \lbrack
0,N-1]}C_{i}.  \label{mainres}
\end{equation}
\end{theorem}

To apply the theorem we need to estimate the quantities related to the
assumptions a), b), c).

Item a) can be obtained by the some approximation inequality showing that $%
L_{\delta }$ well approximates $L$ and an estimation for the norm of $f$
which can be recovered by the Lasota Yorke inequality.

In the following subsection we will prove a Lasota Yorke
inequality for the kind of maps we are interested in (explicitly estimating
its coefficients) involving the $L^{1}$ and bounded variation norm. This
allows to estimate $||f||_{BV}$.

We then estimate (see \cite{GN} Lemma 10) 
\begin{equation*}
||L_{\delta }f-Lf||_{L^1}\leq 2\delta ||f||_{BV}.
\end{equation*}

By this we complete the estimations needed for the first item.

About b), the required $N$ is obtained by the rate of contraction of $%
L_{\delta }$ on the space of zero average measures and will be computed
while running the algorithm by the computer (see \cite{GN} \ for the
details).

Item c) also depend on the definition of $L_{\delta }$; in the case of $%
L^{1} $ approximation with the Ulam method they can be bounded by $1$.

For more details on the implementation of the algorithm, see \cite{GN} .

\paragraph{Lasota Yorke inequality with infinite derivative\label{LY}}

In the following, we see the estimations which are needed to bound the
coefficients of the Lasota Yorke inequality when the map has infinite
derivative.

Let us consider a class of maps which are locally expanding but they can be
discontinuous at some point.

\begin{definition}
We call a nonsingular function $T:([0,1],m)\rightarrow ([0,1],m)$ piecewise
expanding if

\begin{itemize}
\item There is a finite set of points $d_{1}=0,d_{2},...,d_{n}=1$ such that $%
T|_{(d_{i},d_{i+1})}$ is $C^{2}$.

\item $\inf_{x\in \lbrack 0,1]}|D_{x}T|=\lambda ^{-1}>2$ on the set where it
is defined.
\end{itemize}
\end{definition}

Let us define a notion of bounded variation for measures: let 
\begin{equation*}
||\mu ||_{BV}=\underset{\phi \in C^{1},|\phi |_{\infty }=1}{\sup |\mu (\phi
^{\prime })|}
\end{equation*}%
this is related to the usual notion of bounded variation for densities%
\footnote{%
Recall that the variation of a function $g$ is defined as 
\begin{equation*}
\text{Var}(g)=\sup_{(x_{i})\in \text{Finite subdivisions of $[0,1]$}%
}\sum_{i\leq n}|g(x_{i})-g(x_{i+1})|.
\end{equation*}%
}: if $||\mu ||_{BV}<\infty $ then $\mu $ is absolutely  with bounded variation density (see \cite{L2}).

If $f$ is a $L^1$ density, by a small abuse of notation, let us identify it with the associated measure. The following relates the above defined norm with the usual notion of variation

\begin{lemma}
\label{Lemma:BV} Let $f$ be a bounded variation density, then 
\begin{equation*}
||f||_{BV}\leq \text{Var}(f)+2 ||f||_{L^1}
\end{equation*}
\end{lemma}

\begin{proof}
Let $\phi \in C^{1},|\phi |_{\infty }=1$; let $\tilde{\phi}%
=(\phi(1)-\phi(0))\cdot x$ and $\phi_0=\phi-\tilde{\phi}$. Then 
\begin{equation*}
|\int_0^1 \phi^{\prime }f dm|=|\int \phi_0^{\prime }f dm+\int
(\phi(1)-\phi(0)) f dm|\leq |\int \phi_0^{\prime }f dm|+2|\phi|_{\infty}\int
|f|dm,
\end{equation*}
and, as $\phi$ varies, by integration by parts we have that 
\begin{equation*}
||f||_{BV}\leq \text{Var}(f)+2 ||f||_{L^1}.
\end{equation*}
\end{proof}

The following inequality can be established (see \cite{GN}) showing that for
piecewise expanding maps the associated transfer operator is regularizing if
one consider the a suitable norm.

\begin{theorem}
\label{th8} If $T$ is piecewise expanding as above and $\mu $ is a measure
on $[0,1]$ 
\begin{equation*}
||L\mu ||_{BV}\leq \frac{2}{_{\inf T^{\prime }}}||\mu ||_{BV}+\frac{2}{\min
(d_{i}-d_{i+1})}\mu (1)+2\mu (|\frac{T^{\prime \prime }}{(T^{\prime })^{2}}%
|).
\end{equation*}
\end{theorem}

To use the above result in a computation, the problem is that $\mu (|\frac{%
T^{\prime \prime }}{(T^{\prime })^{2}}|)$ cannot be estimated without having
some information on $\mu $. Hence some refinement is necessary. Remark that
if $\mu $ has density $f$ then $||\mu ||_{BV}\geq 2||f||_{\infty }.$

To estimate $\mu (|\frac{T^{\prime \prime }}{(T^{\prime })^{2}}|)$ we
consider $I_{l}=\{x~s.t.~|\frac{T^{\prime \prime }}{(T^{\prime })^{2}}|\geq
l\}$. Let \ $f$ be the density of $\mu $ 
\begin{eqnarray*}
\mu (|\frac{T^{\prime \prime }}{(T^{\prime })^{2}}|) &=&\int_{I-I_{l}}|\frac{%
T^{\prime \prime }}{(T^{\prime })^{2}}|fdx+\int_{I_{l}}|\frac{T^{\prime
\prime }}{(T^{\prime })^{2}}|fdx\leq \\
&\leq &l\int_{I-I_{l}}fdx+||f||_{\infty }\int_{I_{l}}|\frac{T^{\prime \prime
}}{(T^{\prime })^{2}}|dx. \\
&\leq &l\int_{I-I_{l}}fdx+\frac{1}{2}||\mu ||_{BV}\int_{I_{l}}|\frac{%
T^{\prime \prime }}{(T^{\prime })^{2}}|dx
\end{eqnarray*}

If $l$ is chosen such that $\frac{1}{2}\int_{I_{l}}|\frac{T^{\prime \prime }%
}{(T^{\prime })^{2}}|+\frac{2}{_{\inf T^{\prime }}}=\lambda_1 <1$ then we have
the Lasota Yorke inequality which can be used for our purposes:%
\begin{equation}
||L\mu ||_{BV}\leq \lambda_1 ||\mu ||_{BV}+\frac{2}{\min (d_{i}-d_{i+1})}\mu
(1)+l\mu (1).
\end{equation}

\begin{remark}\label{lyrem}

We remark that once an inequality of the form
\begin{equation*}
||Lg||_{{\cal B} '}\leq 2\lambda_1 ||g||_{{\cal B} '}+B^{\prime }||g||_{\cal B}.
\end{equation*}
 is established (with $2\lambda_1 <1$) then, iterating, we have
\begin{equation*}
||L^{n}g||_{{\cal B} '}\leq 2^{n}\lambda_1 ^{n}||Lg||_{{\cal B} '}+\frac{1}{1-2\lambda }B^{\prime
}||g||_{\cal B}
\end{equation*}%
and the coefficient
\[B=\frac{1}{1-2\lambda_1 }B^{\prime
}\] 
can be used to bound the norm of the invariant measure.
\end{remark}

\subsection{An approximation inequality\label{appineq}}

Here we prove an inequality which is used in Remark \ref{r6}.

\begin{lemma}
\label{lemp}For piecewise expanding maps, if $L_{\delta }$ is a Ulam
discretization of size $\delta $, for every measure $f$ \ having \ bounded
variation we have that 
\begin{equation*}
||(L-L_{\delta })f||_{L^1}\leq \delta (2\lambda _{1}+1)||f||_{BV}+\delta
B'||f||_{L^1}
\end{equation*}
\end{lemma}
Where $B'$ is the second coefficient of the Lasota Yorke Inequality related to the map.

\begin{proof}
It holds 
\begin{equation*}
||(L-L_{\delta })f||_{L^1}\leq ||\mathbf{E}(L(\mathbf{E}(f|\mathcal{F}_{\delta
})|\mathcal{F}_{\delta }))-\mathbf{E}(Lf|\mathcal{F}_{\delta }))||_{L^1}+||%
\mathbf{E}(Lf|\mathcal{F}_{\delta })-Lf||_{L^1},
\end{equation*}%
But 
\begin{equation*}
\mathbf{E}(L(\mathbf{E}(f|\mathcal{F}_{\delta })|\mathcal{F}_{\delta }))-%
\mathbf{E}(Lf|\mathcal{F}_{\delta }))=\mathbf{E}[L(\mathbf{E}(f|\mathcal{F}%
_{\delta })-f)|\mathcal{F}_{\delta }].
\end{equation*}%
Since both $L$ and the conditional expectation are $L^{1}$ contractions 
\begin{equation*}
\mathbf{E}(L(\mathbf{E}(f|\mathcal{F}_{\delta })|\mathcal{F}_{\delta }))-%
\mathbf{E}(Lf|\mathcal{F}_{\delta }))\leq ||\mathbf{E}(f|\mathcal{F}_{\delta
})-f||_{L^1}.
\end{equation*}

For \ a bounded variation measure $f$ it is easy to see that $\ ||\mathbf{E}%
(f|\mathcal{F}_{\delta })-f||_{L^1}\leq \delta \cdot ||f||_{BV}.$

By this 
\begin{equation*}
\mathbf{E}(L(\mathbf{E}(f|\mathcal{F}_{\delta })|\mathcal{F}_{\delta }))-%
\mathbf{E}(Lf|\mathcal{F}_{\delta }))\leq \delta ||f||_{BV}.
\end{equation*}

On the other hand%
\begin{equation*}
||\mathbf{E}(Lf|\mathcal{F}_{\delta })-Lf||_{L^1}\leq \delta ||Lf||_{BV}\leq
\delta (2\lambda _{1}||f||_{BV}+B'||f||_{L^1})
\end{equation*}

which gives 
\begin{equation}
||(L-L_{\delta })f||_{L^1}\leq \delta (2\lambda _{1}+1)||f||_{BV}+\delta
B'||f||_{L^1}  \label{1iter}
\end{equation}
\end{proof}

\end{document}